\documentclass[12pt]{amsart}
\usepackage{amsmath, amsthm, amscd, amsfonts, amssymb, graphicx, color}
\usepackage{hyperref}
\usepackage{array}
\usepackage{graphicx}
\usepackage[capitalize,nameinlink]{cleveref}
\usepackage{amsmath, amsthm, amscd, amsfonts, amssymb, graphicx, color}
\newtheorem{thm}{Theorem}
\newtheorem{rmk}{Remark}
\newtheorem{theorem}{Theorem}[section]
\newtheorem{lemma}[theorem]{Lemma}
\newtheorem{proposition}[theorem]{Proposition}
\newtheorem{corollary}[theorem]{Corollary}
\theoremstyle{definition}
\newtheorem{definition}[theorem]{Definition}
\newtheorem{example}[theorem]{Example}

\newtheorem{remark}[theorem]{Remark}
\numberwithin{equation}{section}
\showoutput
\begin{document}
\setcounter{page}{1}

\title[Rigidity of conformal submersions and quasi-Einstein manifolds]{Rigidity of conformal submersions and quasi-Einstein manifolds}

\author[Atreyee Bhattacharya and Sayoojya Prakash]{Atreyee Bhattacharya and Sayoojya Prakash}

\address{Department of Mathematics, Indian Institute of Science Education and Research Bhopal}
\email{atreyee@math.iiserb.ac.in}
\address{Department of Mathematics, Indian Institute of Science Education and Research Bhopal}
\email{sayoojya21@iiserb.ac.in}
\begin{abstract}
In this paper, we study two notions of rigidity, one of conformal submersions and the other of quasi Einstein manifolds, with an attempt to relate the two notions. Note that a smooth submersion between Riemannian manifolds is called \textit{conformal} if it restricts to a conformal isometry on the horizontal distribution. A conformal submersion is said to be \textit{rigid} if it reduces to a Riemannian submersion up to homothety. On the other hand, quasi-Einstein manifolds are generalizations of Einstein manifolds that are of interest both in Riemannian geometry and theoretical physics. A Riemannian manifold $(M,g)$ is called quasi-Einstein if its Ricci tensor satisfies the identity: $Ric_{g} + Hess(f) - \frac{1}{m}df\otimes df = \lambda g$ for some $f \in C^{\infty}(M)$ and constants $\lambda \in \mathbb{R}$ and $0 < m \leq \infty$. A quasi-Einstein manifold is said to be \textit{rigid} if it reduces to an Einstein manifold. In this paper, we employ certain techniques involving conformal submersions to establish rigidity results for a class of closed quasi-Einstein manifolds with $\lambda > 0$.  In particular, we study curvature conditions that force conformal submersions to be rigid, also leading to the rigidity of a related class of quasi-Einstein manifolds.
\end{abstract} 
\maketitle
\section{Introduction}\label{S1}  
In Riemannian geometry, understanding the \textit{rigidity} of geometric properties (e.g. curvature properties, existence of special maps, etc.) has been a standard technique in identifying when a given manifold admits a special Riemannian metric. In this paper, we adhere to the following interpretation of rigidity: A property $\mathcal{P}$ of a Riemannian manifold $(M,g)$ is \textit{rigid} if it has one of the following features:
\begin{enumerate}
\item The existence of $\mathcal{P}$ imposes certain constraints on the underlying topology of $M$ i.e., $(M,g)$ fails to possess $\mathcal{P}$ unless it has the desired underlying structure.
\item $\mathcal{P}$ generalizes another property $\mathcal{P}'$ and under suitable conditions, the existence of both $\mathcal{P}$ and $\mathcal{P}'$ are equivalent.  
\end{enumerate}
The above two notions will be referred to as Type (1) and Type (2) rigidities respectively. Clearly, Type (1) rigidity is a stronger notion 
(that includes constant sectional curvature, Ricci curvature bounded below by a positive constant, etc.) as it restricts the underlying 
topology of a manifold. However, in this paper we address two different Type (2) rigidities: the rigidity of quasi-Einstein manifolds and 
the rigidity of conformal submersions, as elaborated below and try to analyze when one of them implies the other.

Based on curvature properties, complete Riemannian manifolds are broadly classified into three categories: the space forms (manifolds with constant sectional curvature), locally symmetric spaces (manifolds with parallel Riemannian curvature), and Einstein manifolds (manifolds with constant Ricci curvature). While the topology of locally symmetric spaces (that includes space forms) is well known (\cite{ECSI, ECSII}), the conditions under which a given manifold admits an Einstein metric are not fully understood. In order to create non-trivial examples of Einstein metrics, one appeals to Riemannian submersion techniques, which include warped product metrics (\cite{BEM, BBWP}). It turns out that warped product Einstein metrics are in one-to-one correspondence with a special class of Riemannian manifolds (see \cite{CEWP}), known as quasi-Einstein manifolds, that generalizes Einstein manifolds. More precisely, a Riemannian manifold $(M,g)$ is said to be quasi-Einstein if its Ricci curvature $Ric_g$ satisfies the identity:  $\ Ric_g-\lambda g=   \frac{1}{m} df \otimes df -Hess_{g}(f),$ for some $f\in C^{\infty}(M)$, and constants $\lambda \in \mathbb{R}$ and $0 < m \leq \infty$. When $m = \infty$, it is called a Ricci soliton, and when $m \in \mathbb{N}$, it is in one-to-one correspondence with warped product Einstein metrics. A quasi-Einstein metric is said to be \textit{rigid} if it reduces to an Einstein metric. 

 Rigidity of quasi-Einstein manifolds has been of interest both in Riemannian geometry and theoretical physics. In dimension two, R.~Hamilton proved that compact Ricci solitons are rigid (\cite{RHS}). In dimension three, T.~Ivey showed that compact Ricci solitons are rigid, and he extended this result by proving that all compact Ricci solitons with $\lambda \leq 0$ are rigid in all dimensions (\cite{RSTV}). Much later, J.~Case, Y.~Shu, and G.~Wei proved that two-dimensional compact quasi-Einstein manifolds with finite $m$ are rigid. More generally, they established that all compact K\"ahler quasi-Einstein manifolds with finite $m$ are rigid in all dimensions (\cite{RQE}). Moreover, D.S.~ Kim and Y.H.~ Dim showed that compact quasi-Einstein manifolds with finite $m$ and $\lambda \leq 0$ are rigid (\cite{CEWP}). In a series of papers, P.~Petersen and W.~Wylie presented several rigidity results. They proved that a compact \textit{shrinking} gradient Ricci soliton is rigid if and only if $\int_{M}Ric(\nabla f, \nabla f) dV_{g} \leq 0$ (\cite{RGRSPW}). They also showed that homogeneous gradient Ricci solitons are rigid (\cite{GRSPW}). Furthermore, they established that simply connected shrinking gradient Ricci solitons with vanishing Weyl tensor and Ricci tensor satisfying a weak integral condition are isometric to $\mathbb{S}^{n}$, $\mathbb{S}^{n-1} \times \mathbb{R}$, and $\mathbb{R}^{n}$ (\cite{CGSPW}). Moreover, A.~Naber proved that four-dimensional noncompact shrinking Ricci solitons with a bounded non-negative curvature operator are isometric to $\mathbb{R}^{4}$ or finite quotients of  $\mathbb{S}^{2} \times \mathbb{R}^{2}$ or $\mathbb{S}^{3} \times \mathbb{R}$ (\cite{RSAN}). Additionally, G.~Catino, C.~Mantegazza, L.~Mazzieri, and M.~Rimoldi proved that complete locally conformally flat quasi-Einstein manifolds of dimension $n \ge 3$ are locally a warped product with $(n - 1)$-dimensional fibers of constant curvature (\cite{LCFQEM}).
 
In this paper, we employ techniques from conformal submersions to study the rigidity of certain quasi-Einstein metrics: conformal submersions are generalizations of Riemannian submersions where the differential of the submersion map is a conformal isometry when restricted to its horizontal distribution. Since Riemannian submersions have been instrumental in constructing new Einstein metrics, one would naturally expect conformal submersion techniques to play a similar role for quasi-Einstein metrics. A conformal submersion is said to be \textit{rigid} if it reduces to a Riemannian submersion up to homothety. Recently, the rigidity of conformal submersions was established for a special class of \textit{contact metric manifolds} (\cite{RSAD}) using standard contact Riemannian geometry techniques. However, in general, not much is known about the rigidity of conformal submersions that the authors are aware of. We connect the notions of rigidity of conformal submersions and rigidity of quasi-Einstein metrics as follows.  We begin by studying curvature conditions that impose rigidity on conformal submersions where the total space, fibers, and base are all closed, orientable Riemannian manifolds, and the total space is of dimension at least four. From these results, we obtain a rigidity criterion for certain quasi-Einstein manifolds. Recall that a Riemannian submersion is described by its fundamental tensors $T$ and $A$, respectively measuring the obstruction to the fibers being totally geodesic and the integrability of the horizontal distribution (see \cite{BEM} for details). Analogously, for a conformal submersion with conformal factor $e^{2f}$, one can define fundamental tensors, denoted by $T^{f}$ and $A^{f}$, which depend on $f$. While $A^{f}$ represents the obstruction to the integrability of the horizontal distribution of the conformal submersion, in this case, the vanishing of $T^{f}$ implies that the fibers are totally umbilical (see \cref{s22} for details). We prove the following result.

\begin{thm}\label{T11}
 Let $\pi_{f} \colon (M,g) \rightarrow (B,g_{B})$ be a conformal submersion with fiber $(F,g_F)$ such that $(M,g)$ is locally conformally flat and  $dim(B) = dim(F)$. If the scalar curvatures of the total space and fibers have particular signs (not necessarily the same sign) and $(F,g_{F},h,m,\lambda)$ is quasi-Einstein, then $\pi_{f}$ is rigid if either of the following conditions holds.
\begin{enumerate}
\item $T^{f} = 0$, the scalar curvature of the base is non-positive, and $\lambda \leq 0$.
\item The horizontal distribution is integrable, $T^{f}$ is parallel with respect to $g$, the scalar curvature of the base is non-negative, and $\lambda \geq 0$.
\end{enumerate}
\end{thm}

\begin{rmk}\label[remark]{R12}
\begin{enumerate}
    \item Our analysis mostly uses classical techniques from Riemannian geometry, such as the weak maximal principle and the divergence theorem. 
    \item We also provide criteria for rigidity of conformal submersions similar to \cref{T11} by assuming either the total space or the base to be \textit{quasi-Einstein} (see Theorems \ref{T35} and \ref{T38}). 
    \item The condition $dim(F) = dim(B)$ in \cref{T11} is necessary as shown in \cref{E310}.
\end{enumerate}
\end{rmk}
The next result deals with the rigidity of conformal submersions where the total space is the product of closed Einstein manifolds.
\begin{thm}\label{T13}
Let $\pi_{f} \colon (M,g) \rightarrow (B,g_{B})$ be a conformal submersion with fiber $F$, where the total space $(M,g)$ 
is the product of closed Einstein manifolds. 
Then $\pi_{f}$ is rigid if it satisfies one of the following three conditions.
\begin{enumerate}
\item The horizontal distribution is integrable, $tr(T^{f})$ (with respect to $g$) is parallel with respect to $g$, together with  fibers having constant scalar curvature (with the same constant) and $dim(B) \geq 2$,
\item  The trace $tr(T^{f})$ (with respect to $g$) vanishes together with either of the following conditions: 
\begin{enumerate}
\item The fibers have constant scalar curvature (with the same constant), $\dim (B) \ge 2$, and  $|A^{f}|_g^{2}$ is constant,
\item $(B,g_{B})$ is scalar-flat and $ 2|A^{f}|_g^{2} + |T^{f}|_g^{2}$ is constant,
\end{enumerate}
\item The fibers are totally geodesic with respect to $g$, $(B,g_{B})$ is scalar-flat with $\dim (B) \ge 2$, and $|A^{f}|_g^{2}$ is constant.   
\end{enumerate}
\end{thm}
As an immediate consequence, we arrive at a conclusion similar to Corollary 9.64 of \cite{BEM}.
\begin{rmk}
    If $\pi_{f} : (M,g) \rightarrow (B,g_{B})$ is a principal G-bundle over a scalar flat base with totally geodesic fibers with respect to $g$, and if  $\pi_{f}$ is a non-trivial conformal submersion with constant $|A^{f}|^{2}$, then $M$ cannot admit an Einstein metric.
\end{rmk}
\begin{rmk}
    Note that, \cref{T13}(3) is similar to Corollary 9.62 of \cite{BEM}.
\end{rmk}
In this paper, we obtain rigidity criteria for closed quasi-Einstein manifolds by realizing it as the base or fiber of a suitable conformal submersion and then using the aforementioned rigidity criteria for conformal submersions. We consider quasi-Einstein manifolds of dimensions at least $3$ and prove the following result.
\begin{thm}\label{T14}
A closed orientable quasi-Einstein manifold $(M,g,h,m,\lambda)$ is rigid if it satisfies one of the following criteria. 
\begin{enumerate}
\item $m \ge 1$, $\lambda \ge 0$, and $(M,g)$ is the fiber (respectively, the base) of a conformal submersion with 
integrable horizontal distribution, parallel $T^{f}$ (with respect to $g$), locally conformally flat total space having twice the dimension of $M$, and the base (respectively, the fiber) has non-negative scalar curvature. 
    \item $(M,g)$ is the base of a conformal submersion with integrable horizontal distribution, parallel $tr(T^{f})$ (with respect to $g$) and $|T^{f}|_g^{2}$ is constant, where the total space is the product of Einstein manifolds, and fibers have constant scalar curvature (with the same constant).  
    \item  $(M,g)$ is the base of a conformal submersion with vanishing $tr(T^{f})$ and constant $|T^{f}|_g^{2}$ and $|A^{f}|_g^{2}$, where the total space is the product of Einstein manifolds, and fibers have constant scalar curvature (with the same constant). 
\end{enumerate}    
\end{thm}
\begin{rmk}\label{R15}
We obtain rigidity results for a quasi-Einstein manifold analogous to \cref{T14} (2),(3) by realizing it as the fiber of a conformal submersion (see \cref{T56} and \cref{T57}). Additionally, we establish rigidity results for a quasi-Einstein manifold similar to \cref{T14} (1) by realizing it as the total space of a conformal submersion (see \cref{T52}) and conclude the following result (see \cref{E53} for details).  
\end{rmk}
\begin{thm}
    Let $(M,g)$ be a closed locally conformally flat manifold which is a Riemannian product of manifolds of equal dimension. If $(M,g,f,m,\lambda)$ is quasi-Einstein with $m \geq 1$ up to conformal equivalence, then $(M,g)$ is rigid. In particular, there does not exist any non-trivial even-dimensional closed locally conformally quasi-Einstein manifold with $m \geq 1$ that splits into a Riemannian product of two manifolds of equal dimension up to conformal equivalence.
\end{thm}
%\begin{remark}
 %   In particular, if $(M,g,f,m,\lambda)$ is a locally conformally flat quasi-Einstein manifold with $1 \leq m < \infty$ and $\lambda > 0$ up to conformal class, then $M$ and its universal cover will not split into two same-dimensional manifolds.
%\end{remark}
%\begin{theorem}\label{T16}
%Let $(M,g,h,m,\lambda)$ be a closed locally conformally flat quasi-Einstein manifold with $m \geq 1$ and $\lambda \geq 0$ and of dimension $\geq 4$. If the universal cover of $(M,g)$ is closed and the Riemannian product of two manifolds of equal dimensions, then $(M,g)$ is flat.    
% \end{theorem}
\begin{rmk} \label[remark]{R17}
Using \cref{T13}, we partially recover Theorem 1* of \cite{CDKR} as stated below:
Let $(M,g)$ be a closed Einstein manifold of dimension at least $4$, which is the total space of a conformal submersion 
$\pi\colon (M,g) \rightarrow (B,g_{B})$ with integrable horizontal distribution and vanishing $T^{f}$. Then up to homothety, 
$g$ is the unique Einstein metric in its conformal class whenever the fibers are one-dimensional or
the base is scalar-flat.   
\end{rmk}
\tableofcontents
\section{Preliminaries}\label{S2}
\noindent In this section, we recollect some basic concepts from Riemannian geometry that are relevant to this paper (see \cite{BEM, RGGHL, RGPP} for details) and set up the notation to be followed. Let $(M,g)$ denote a Riemannian manifold and $\nabla$ the associated Levi-Civita connection. We will restrict to closed (compact without boundary), connected, and orientable Riemannian manifolds. 
\begin{definition}\label{D21}   
Given a Riemannian manifold $(M,g)$, its \textbf{Riemann curvature} $R$ is the $(0,4)$ tensor defined by
$$R(X,Y,Z,W) = g(\nabla_{Y}\nabla_{X}Z - \nabla_{X}\nabla_{Y}Z + \nabla_{[X,Y]}Z,W),$$
where $X,Y,Z$ and $W$ are arbitrary vector fields defined on an open subset of $M$.
\end{definition}
\begin{remark}\label{R22}
The Riemann curvature tensor of $(M,g)$ satisfies the following orthogonal decomposition (cf. \cite{RGGHL}) 
\begin{equation}\label{Eq21}
    R =  \frac{s}{2n(n-1)}g \circ g + \frac{1}{n - 2}(Ric - \frac{s}{n}g) \circ g + W
\end{equation}
where $Ric$, $W$ and $s$ denote the Ricci, Weyl and scalar curvatures of $(M,g)$ respectively and $\circ$ denotes the  \textit{Kulkarni-Nomizu product} for symmetric $(0,2)$ tensors of $(M,g)$ defined by,
\begin{eqnarray*}
        (h\circ k)(X,Y,Z,W) = &&h(X,Z)k(Y,W) + h(Y,W)k(X,Z)
        - h(X,W)k(Y,Z)\\ &&- h(Y,Z)k(X,W), \ \ \ \forall\,h,\,k\,\in\,S^{2}(M).
\end{eqnarray*}
\end{remark}

 \begin{definition}\label{D23}
 A complete Riemannian manifold $(M,g)$ is said to be  
\begin{enumerate}
\item 
 a \textbf{Space form} if it has constant sectional curvature.
\item
\textbf{Einstein} if its Ricci curvature is a scalar multiple of the Riemannian metric, i.e., $Ric_{g} = \lambda g \ \ \text{ for some } \ \lambda \in \mathbb{R}.$
\item \textbf{locally conformally flat} if given any point $p \in M$, there exists an open set $\mathcal{U}$ containing $p$ and $f \in C^{\infty}(\mathcal{U})$ such that the metric $e^{2f}g|_{\mathcal{U}}$ is flat.
     \end{enumerate}
 \end{definition}
 \begin{remark}\label{R24}
A Riemannian manifold $(M
,g)$ of dimension at least $4$ is locally conformally flat if and only if its Weyl tensor vanishes (cf. \cite{RGGHL})
 \end{remark}
 \subsection{Quasi-Einstein manifolds.}\label{s21}
In this section, we briefly discuss a special class of Riemannian manifolds known as quasi-Einstein manifolds, which are certain generalizations of Einstein manifolds. We refer to \cite{RQE, LCFQEM, CGEGY} for detailed discussions on quasi-Einstein manifolds. 
\begin{definition}\label{D25}
 A Riemannian manifold $(M,g)$ is said to be
\begin{enumerate}
    \item a quasi-Einstein manifold if its Ricci curvature satisfies the identity:
    $$Ric_{g} + Hess_{g}(f) - \frac{1}{m}df \otimes df = \lambda g$$
    where $f \in C^{\infty}(M)$, $\lambda \in \mathbb{R}$, $0 < m \leq \infty$ and $Hess_{g}(f)(X,Y) = g(\nabla_{X}\nabla f,Y)$.
    \item a Ricci soliton if there exists a vector field $X \in \chi(M)$ such that
    $$Ric_{g} + \frac{1}{2}L_{X} g = \lambda g,$$
    where $\lambda \in \mathbb{R}$  and $L_{X}g$ is the Lie-derivative of $g$ with respect to $X$. In particular, if $X = \nabla f$ for some $f \in C^{\infty}(M)$, then $(M,g)$ is called a \textbf{gradient Ricci soliton}.
\end{enumerate}
\end{definition}
\begin{remark}\label{R26}
    If $m = \infty$, then a quasi-Einstein manifold $(M,g,f)$ becomes a gradient Ricci soliton, and if $m \in \mathbb{N}$, then $(M,g,f)$ corresponds to a warped product Einstein metric (\cite{CEWP}).
\end{remark}
\begin{definition}\label{D27}
A quasi-Einstein manifold $(M,g,f,m,\lambda)$ is said to be \textbf{rigid} if it reduces to an Einstein manifold. 
\end{definition}
\begin{remark}\label{R28}
    For a closed, oriented, connected, quasi-Einstein manifold, being rigid is equivalent to $f$ being a constant. This is not true for noncompact manifolds (see Proposition 4.2 of \cite{RQE}).
\end{remark}
\subsection{Riemannian and conformal submersions}\label{s22}
\hfill \break
Let $(M,g)$ and $(B,g_B)$ be Riemannian manifolds and $\pi: M \rightarrow B$ a smooth submersion. Given $p \in M$ and $\pi(p) = b$, let $F_{b}=\pi^{-1}(b)$ denote the \textit{fiber at $b$}, $\mathcal{V}_{p}:= Ker(d\pi_{p}) = T_{p}F_{b}$, and let $\mathcal{H}_{p}$ be the orthogonal complement of $\mathcal{V}_{p}$ in $T_{p}M$. \\
The distributions $\mathcal{V}: p \mapsto \mathcal{V}_{p}, \ p \in M$, and $\mathcal{H}: p \mapsto \mathcal{H}_{p}, \ p \in M$ are called the \textit{vertical} and \textit{horizontal distributions} of $\pi$ respectively. 
\begin{definition}\label{D29}
 A (smooth) submersion $\pi \colon (M,g) \rightarrow (B,g_{B})$ as above is called 
\begin{enumerate}
    \item a \textbf{Riemannian submersion} if $g|_{\mathcal{H}} = \pi^{*}g_{B}$.
   \item  a \textbf{conformal submersion} if $g|_{\mathcal{H}} = e^{2f}\pi^{*}g_{B}$ for some $f \in C^{\infty}(M)$. In such a case, we will denote the submersion map $\pi$ by $\pi_f$.
\end{enumerate}
\end{definition}
\begin{remark}\label{R210}
A vector field $\xi$ on $M$ is called \textit{vertical} (respectively, \textit{horizontal}) if $\xi_{p} \in \mathcal{V}_{p}$( respectively, $\mathcal{H}_{p}$), $\forall\, p \in M$.
\end{remark}
\begin{definition} \label{D211}
Given a Riemannian submersion $\pi\colon(M,g) \rightarrow (B,g_{B})$, the \textbf{fundamental tensors} $T$ and $A$ are defined as follows, 
\begin{equation}\label{Eq22}
\begin{aligned}
T_{E}F = \mathcal{H}\nabla_{\mathcal{V}E}\mathcal{V}F + \mathcal{V}\nabla_{\mathcal{V}E}\mathcal{H}F\\
A_{E}F = \mathcal{H}\nabla_{\mathcal{H}E}\mathcal{V}F + \mathcal{V}\nabla_{\mathcal{H}E}\mathcal{H}F
\end{aligned}
\end{equation}
where $E,F \in \chi(M)$, $\nabla$ is the Levi-Civita connection of $(M,g)$ and $\mathcal{H}E$ and $\mathcal{V}E$ respectively denote the horizontal and vertical components of $E$.
\end{definition}  
\begin{definition}\label{D212}
    Given a conformal submersion $\pi_f \colon (M,g) \rightarrow (B,g_{B})$ as in \cref{D29},  consider the metric $g_{\pi}$ on $M$ given by $g_{\pi} = e^{-2f}g.$ Then $\pi_{f} \colon (M,g_{\pi}) \rightarrow (B,g_{B})$ defines a Riemannian submersion.  $T^{f}$ 
and $A^{f}$ denote the fundamental tensors of the Riemannian submersion $\pi_{f}: (M,g_{\pi}) \rightarrow (B,g_{B})$.
\end{definition}
\begin{definition}\label{D213}
A conformal submersion $\pi_{f} \colon (M,g) \rightarrow (B,g_{B})$ is said to be \textbf{rigid} if it reduces to a Riemannian submersion up to homothety, i.e., $f$ is a constant function.
\end{definition}
\subsubsection{\textbf{Structural equations for conformal submersions.}}
\hfill \break
Given a conformal submersion $\pi_f \colon (M,g) \rightarrow (B,g_{B})$ as in \cref{D29}, 
and the metric $g_{\pi} = 
e^{-2f}g$ on $M$, let $\nabla, \tilde{\nabla}$ denote the Levi-Civita connections of $(M,g)$ and $(M,g_{\pi})$ respectively. 
Let $U, V$ denote vertical vector fields and $X, Y$ denote horizontal vector fields defined in a neighbourhood of an arbitrary point $p \in M$. Then, the following structural equations hold:
 \begin{equation}\label{Eq23}
 \begin{aligned}
 \mathcal{H}\nabla_{U}V =& T^{f}_{U}V - \mathcal{H}(\nabla f)g(U,V) \\
\mathcal{V}\nabla_{U}V =& \mathcal{V}\tilde{\nabla}_{U}V + U(f)V + V(f)U - \mathcal{V}(\nabla f)g(U,V)\\
\mathcal{H} \nabla_{X}U =& A^{f}_{X}U + U(f)X\\
\mathcal{V} \nabla_{X}U =& \mathcal{V} \tilde{\nabla}_{X}U + X(f)U\\
\mathcal{H}\nabla_{U}X =& \mathcal{H}\tilde{\nabla}_{U}X + U(f)X \\
\mathcal{V}\nabla_{U}X =& T^{f}_{U}X +  X(f)U\\
\mathcal{H}\nabla_{X}Y =& \mathcal{H}\tilde{\nabla}_{X}Y + X(f)Y + Y(f)X - \mathcal{H}(\nabla f)g(X,Y)\\
\mathcal{V}\nabla_{X}Y =& A^{f}_{X}Y - \mathcal{V}(\nabla f)g(X,Y).
 \end{aligned}
 \end{equation}
 \begin{definition}\label{D214}
Let $\pi_{f}  \colon (M,g) \rightarrow (B,g_{B})$ be a conformal submersion. The fibers of $\pi_{f}$ are said to be \textbf{totally umbilical} (that is, all fibers are totally umbilical submanifolds of $(M,g)$) if there exists a horizontal vector field $X$ such that $\mathcal{H}\nabla_{U}V = g(U, V)X$ for all vertical vector fields $U, V$. The geometric interpretation of totally umbilical fibers is that the fibers bend uniformly in all directions (cf. \cite{ONB}).
\end{definition}
 Recall that Riemannian submersion is characterized by its fundamental tensors $T$ and $A$.  Analogously, the fundamental tensors of a conformal submersion are denoted by $T^{f}$ and $A^{f}$. In a Riemannian submersion, $T$ vanishes if and only if the fibers are totally geodesic. In contrast, for a conformal submersion, the vanishing of $T^{f}$ implies that the fibers are totally umbilical. Moreover, in a conformal submersion, the fibers are totally geodesic if and only if $T^{f}_{U}V = g(U, V)\mathcal{H}(\nabla f)$. In both Riemannian and conformal submersion, the tensors $A$ and $A^{f}$ vanish if and only if the horizontal distribution is integrable. 
\begin{lemma}\label{L215}
    Let $p \in M$ and $\mathcal{U}$ a neighborhood of $p$ in $M$. Let $U, V$ be vertical vector fields and $X, Y$ horizontal vector fields defined in $\mathcal{U}$ with the conditions $|U| = |V| = |X| = |Y| = 1, |U \wedge V| = |X \wedge Y| = 1$. Let $K_{p}, \check{K}_{p}, \hat{K}_{\pi_{f}(p)}$ respectively denote the sectional curvatures of $(M,g)$ and $(F_{\pi_{f}(p)},g_{F_{\pi_{f}(p)}})$ at $p$ and $(B,g_{B})$ at $\pi_{f}(p)$. The following relations hold.
\begingroup\makeatletter\def\f@size{11}\check@mathfonts
\def\maketag@@@#1{\hbox{\m@th\normalfont#1}}
\begin{align}\label{Eq24}
K_{p}(U,V) =& \check{K}_{p}(U,V) - g_{p}(T^{f}_{V}V,T^{f}_{U}U) + g_{p}(T^{f}_{U}V,T^{f}_{U}V)\\& + T^{f}_{U}U(f)(p) + T^{f}_{V}V(f)(p) - |\mathcal{H}(\nabla f)|_{p}^{2} \nonumber   
\end{align}
\begin{align}\label{Eq25}
K_{p}(X,U) =& 2X(f)(p)g_{p}(T^{f}_{U}U,X) + g_{p}((\nabla_{X}T^{f})_{U}U,X) + |A^{f}_{X}U|_{p}^{2}\nonumber\\& -|T^{f}_{U}X|_{p}^{2} + (\mathcal{H}\nabla_{X}X)(f)(p) - (X(f))^{2}(p) - X^{2}(f)(p)\\&   + (\mathcal{V}\nabla_{U}U)(f)(p) - (U(f))^{2}(p) - U^{2}(f)(p) \nonumber
\end{align}
\begin{align}\label{Eq26}
K_{p}(X,Y) =& e^{-2f(p)}\hat{K}_{\pi_{f}(p)}(X,Y) - 3|A^{f}_{X}Y|_{p}^{2}+ |\nabla f|_{p}^{2} - Hess f(X,X)(p)\\& - (X(f))^{2}(p) -  Hess f(Y,Y)(p)- (Y(f))^{2}(p)\nonumber  \end{align}
\endgroup
\end{lemma}
\begin{proof}
     Let $U, V, W, W'$ be vertical vector fields and $X, Y, Z, Z'$ horizontal vector fields defined in $\mathcal{U}$. Let $R$, $\check{R}$, and $\hat{R}$ denote the Riemannian curvature tensors of $(M,g)$, $(F_{\pi_{f}(p)},g_{F_{\pi_{f}(p)}})$, and $(B,g_{B})$ respectively. Then the curvature tensors are given below 
     \begingroup\makeatletter\def\f@size{11}\check@mathfonts
\def\maketag@@@#1{\hbox{\m@th\normalfont#1}}
\allowdisplaybreaks{
\begin{align}\label{Eq27}
  R_{p}(U,V,W,W') =& g_{p}(\check{R}(U,V)W,W') - g_{p}(T^{f}_{V}W',T^{f}_{U}W) + g_{p}(T^{f}_{V}W,T^{f}_{U}W')\\& + (T^{f}_{U}W)(f)(p)g_{p}(V,W') - (T^{f}_{V}W)(f)(p)g_{p}(U,W')\nonumber\\&
 + g_{p}(U,W)T^{f}_{V}W'(f)(p) - g_{p}(V,W)T^{f}_{U}W'(f)(p)\nonumber\\&
 +  (g_{p}(U,W')g_{p}(V,W) - g_{p}(U,W)g_{p}(V,W'))|\mathcal{H}(\nabla f)|_{p}^{2},\nonumber
 \end{align}
\begin{align}\label{Eq28}
R_{p}(X,Y,Z,Z') =& e^{2f(p)}\pi_{f}^{*}(\hat{R}_{\pi_{f}(p)})(X,Y,Z,Z') - 2g_{p}(A^{f}_{X}Y,A^{f}_{Z}Z')\\& + g_{p}(A^{f}_{Y}Z,A^{f}_{X}Z')- g_{p}(A^{f}_{X}Z,A^{f}_{Y}Z')\nonumber\\& -Hess f(X,Z)(p)g_{p}(Y,Z') - (X(f)Z(f))(p)g_{p}(Y,Z')\nonumber\\&  + Hess f(Y,Z)(p)g_{p}(X,Z') + (Y(f)Z(f))(p)g_{p}(X,Z')\nonumber\\&- Hess f(Y,Z')(p)g_{p}(X,Z) - (Y(f)Z'(f))(p)g_{p}(X,Z)\nonumber\\& + Hess f(X,Z')(p)g_{p}(Y,Z) + (X(f)Z'(f))(p)g_{p}(Y,Z)\nonumber\\& + (g_{p}(X,Z)g_{p}(Y,Z')- g_{p}(X,Z')g_{p}(Y,Z))|\nabla f|_{p}^{2}.\nonumber 
\end{align}
\begin{align}\label{Eq29}
    R_{p}(X,U,Y,V) =& -g_{p}((\nabla_{X}T^{f})_{U}Y,V) - g_{p}(T^{f}_{U}X,T^{f}_{V}Y) -X(f)g_{p}(T^{f}_{U}Y,V)\nonumber\\& +(\mathcal{H}\nabla_{X}Y)(f)g_{p}(U,V) + (\mathcal{V}\nabla_{U}V)(f)g_{p}(X,Y) + g_{p}(\nabla_{U}A^{f}_{X}Y,V)\\& -Y(f)g_{p}(T^{f}_{U}X,V) - X(f)Y(f)g_{p}(U,V) - U(f)V(f)g_{p}(X,Y)\nonumber\\& - g_{p}(X,Y)UV(f) -g_{p}(U,V)XY(f) + g_{p}(A^{f}_{X}V,A^{f}_{Y}U)\nonumber\\&  + V(f)g_{p}(A^{f}_{Y}U,X) - U(f)g_{p}(A^{f}_{X}Y,V)\nonumber
\end{align}}
\endgroup
\end{proof}
Using these curvature equations, we can conclude the following results.
\begin{corollary}
Let $p \in M$ and $r_{p},\check{r}_{p}$ denote the Ricci curvatures of $(M,g)$ and $(F_{\pi_{f}(p)},g_{F_{\pi_{f}(p)}})$ at $p$, respectively and the Ricci curvature of $(B,g_{B})$ at $\pi_{f}(p)$ by $\hat{r}_{\pi_{f}(p)}$. Let $\{X_{i}\}_{i\in I}$, $\{U_{j}\}_{j\in J}$ respectively denote local orthonormal frames of the horizontal and vertical distributions at $p$. The following relations hold.
    \begingroup\makeatletter\def\f@size{11}\check@mathfonts
\def\maketag@@@#1{\hbox{\m@th\normalfont#1}}
\allowdisplaybreaks{
\begin{align}\label{Eq210}
r_{p}(U,V) =& \check{r}_{p}(U,V)- g_{p}(N^{f},T^{f}_{U}V) + (\delta T^{f})(U,V)(p) + g_{p}(A^{f}U,A^{f}V) + dim(F)T^{f}_{U}V(f)(p)\nonumber\\& +g_{p}(U,V)N^{f}(f)(p)  - dim(F)g_{p}(U,V)|\mathcal{H}(\nabla f)|_{p}^{2}  -g_{p}(U,V) (\Delta f)|_{\mathcal{H}}(p)\\& + g_{p}(U,V)dim(B)|\mathcal{V}(\nabla f)|_{p}^{2} + dim(B)((\mathcal{V}\nabla_{U}V)(f)- UV(f) -U(f)V(f))(p)\nonumber;
\end{align}
\begin{align}\label{Eq211}
 r_{p}(X,U) =& g_{p}((\check{\delta}T^{f})(U),X) + g_{p}(\nabla_{U}N^{f},X) - g_{p}((\hat{\delta}A^{f})(X),U) - 2g_{p}(T^{f}_{U},A^{f}_{X})\nonumber\\&- (dim(B) - 3)A^{f}_{X}U(f)(p) - (dim(F) -1)T^{f}_{U}X(f)(p)\\&- (dim(M) - 2)(Hessf(X,U) + X(f)U(f))(p);\nonumber
\end{align}
\begin{align}\label{Eq212}
 r_{p}(X,Y) =& \pi_{f}^{*}(\hat{r}_{\pi_{f}(p)})(X,Y) - 2g_{p}(A^{f}_{X},A^{f}_{Y})  - g_{p}(T^{f}X,T^{f}Y)  +  g_{p}(\nabla_{X}N^{f},Y)\nonumber\\& + \sum_{j}g_{p}(\nabla_{U_{j}}A^{f}_{X}Y,U_{j}) - (dim(M) -2)(Hess f(X,Y) + X(f)Y(f))(p)\\& - dim(F)A^{f}_{X}Y(f)(p) + (dim(M) - 2)g_{p}(X,Y)|\nabla f|_{p}^{2}\nonumber\\& + X(f)(p)g_{p}(N^{f},Y) + Y(f)(p)g_{p}(N^{f},X)- g_{p}(X,Y)(N^{f}(f) + \Delta f)(p).\nonumber
\end{align}}
\endgroup
where 
$N^{f} = Trace(T^{f})$, $(\delta T^{f})(U, V) = \sum_{i}g((\nabla_{X_{i}}T^{f})_{U}V, X_{i})$, and for any $(2,1)$ tensor $E$ on $M$\\
$\check{\delta}E = -\sum_{j}(\nabla_{U_{j}}E)_{U_{j}}$, and $\hat{\delta}E = -\sum_{i}(\nabla_{X_{i}}E)_{X_{i}}$.\\
\end{corollary}
\begin{corollary}\label{C217}
The scalar curvatures of $(M,g)$, $(F_{\pi_{f}(p)},g_{F_{\pi_{f}(p)}})$, and $(B,g_{B})$ denoted by $s,\check{s},\hat{s}$ respectively, are related as follows
\begin{align}\label{Eq213}
s &= e^{-2f}\hat{s}\circ \pi_{f} + \check{s}  - |A^{f}|^{2} - |T^{f}|^{2} + div(N^{f}) - \hat{\delta}N^{f}\nonumber\\&+ dimB(dimM + dimF -3)|\nabla f|^{2} + (2-2dimB + dimF)N^{f}(f)\\&-(2dimF -2)\Delta f|_{\mathcal{H}}- 2dimB\Delta f-(dimF -2 + dim(F)^{2})|\mathcal{H}(\nabla f)|^{2},\nonumber
\end{align}
where $\hat{\delta}N^{f} = -\sum_{j}(\delta T^{f})(U_{j},U_{j})$.
\end{corollary}
  \begin{remark}\label{R218}
Here, we will use the sign conventions given below. 
\begin{enumerate}
\item Given a vector field $X$ in a Riemannian manifold $(M,g)$, the divergence $div(X)$ of $X$ is defined point-wise as $div(X) = \sum_{i} g(\nabla_{e_{i}}X,e_{i})$, where $\{e_{i}\}$ is an orthonormal basis at that point.
\item Given $h \in C^{\infty}(M)$, $\Delta h := div(\nabla h)$.
\item Given a conformal submersion $\pi_{f}: (M,g) \rightarrow (B,g_{B})$ and $h \in C^{\infty}(M)$, define $\Delta h|_{\mathcal{H}} := \sum_{i}g(\nabla_{X_{i}}\nabla h,X_{i})$ and $\Delta h|_{\mathcal{V}} := \sum_{j}g(\nabla_{U_{j}}\nabla h,U_{j})$
, where $\{X_{i}\}_{i\in I}$, $\{U_{j}\}_{j\in J}$ respectively denote local orthonormal frames of the horizontal and vertical distributions.
\end{enumerate}
  \end{remark}
\begin{remark}\label{R219}
From now on, we only consider closed-orientable Riemannian manifolds and conformal submersions with total space of dimension at least $4$.
 \end{remark}
 \section{Conformal submersions with locally conformally flat total space}\label{s3}
\noindent Although the title refers to locally conformally flat total spaces, we consider conformal submersions whose total spaces are Riemannian products of 
locally conformally flat and Einstein manifolds (where the Einstein components are possibly trivial), and study their rigidity. We further assume that the base and the fibers have the same dimensions. In such a case, we provide suitable curvature conditions imposing rigidity.
\subsection{The total space is the product of locally conformally flat and Einstein manifolds.}\label{s31} In conformal submersions 
whose total spaces are Riemannian products of locally conformally flat and Einstein manifolds (with possibly no Einstein components), we make the following observations about the scalar curvature of the components of the total space.
\begin{proposition}\label[proposition]{P31}
\begin{enumerate}
\item  Let $\pi_{f}\colon (M,g) \rightarrow (B,g_{B})$ be a conformal submersion such that $(M,g)$ is locally conformally flat and  $dim(B) = dim(F)$ where $F_{p}$ denotes the fiber at $p \in B$. Then the scalar curvature of $(M,g)$ satisfies the following equation.
\begin{equation}\label{Eq31}
\footnotesize
\frac{ns}{2(n-1)}= 2(|A^{f}|^{2} - |T^{f}|^{2} - \hat{\delta}N^{f}) -(n-4)N^{f}(f) - n\Delta f + \frac{n(n-2)}{2}|\nabla f|^{2},
\end{equation}
where $n = dim(M)$.
\item  More generally, consider a conformal submersion $\pi_{f}\colon (M,g) \rightarrow (B,g_{B})$ where $(M,g)$ is the Riemannian product of locally conformally flat manifolds each of dimension at least $4$ such that the horizontal and the vertical distributions restricted to each component of $(M,g)$ have the same dimension. Then,
\begin{equation}\label{Eq32}
\footnotesize
\sum_{i} \frac{n_{i}s_{i}}{2(n_{i} - 1)} = 2(|A^{f}|^{2} - |T^{f}|^{2} - \hat{\delta}N^{f}) -(n-4)N^{f}(f) - n\Delta f + \frac{n(n-2)}{2}|\nabla f|^{2},
\end{equation}
where $n_{i}$ and $s_{i}$ respectively denote the dimension and the scalar curvature of the $i^{th}$ component of $(M,g)$, and $n = dim(M)$.  
\item  Let $\pi_{f}\colon (M \times N,g_{M} + g_{N}) \rightarrow (B,g_{B})$ be a conformal submersion where $(M,g_{M})$ is as in $(2)$.
Let $(N,g_{N})$ be the Riemannian product of Einstein manifolds where the tangent bundle of each component of $(N,g_{N})$ is completely contained either in the vertical or the horizontal distribution of $\pi_f$, and $dim(F) = dim(B)$ where $F_{p}$ denotes the fiber at $p \in B$. Then,
 \begin{equation}\label{Eq33}
 \footnotesize
 \begin{aligned}
\sum_{i} \frac{n_{i}s_{i}}{2(n_{i} - 1)} =& 2(|A^{f}|^{2} - |T^{f}|^{2} - \hat{\delta}N^{f}) -(n-4)N^{f}(f) - (\frac{n + l}{2})(\Delta f - \frac{n-2}{2}|\nabla f|^{2})\\&
- \frac{e}{2}(div(\mathcal{H} \nabla f) - \frac{n-2}{2}|\mathcal{H}\nabla f|^{2}),
\end{aligned}
\end{equation}
where $n_{i}$ and $s_{i}$ are the dimensions and the scalar curvatures of the $i^{th}$ component of $(M,g)$, respectively and $l = dim(M)$, $e = dim(N)$ and $l+e = n$.      
\end{enumerate}
\end{proposition}
\begin{proof}
(1) Let $p \in M$. Let $\{X_{i}\}{i\in I}$, $\{U_{j}\}_{j\in J}$ respectively denote local orthonormal frames of the horizontal and vertical distributions at $p$. Since the Weyl curvature $W = 0$, using the orthogonal decomposition of the Riemann curvature tensor given in \cref{Eq21}, it follows that 
\begin{equation}\label{Eq34}
\sum_{j \neq k}R(U_{k},U_{j},U_{k},U_{j}) = \frac{m -2}{n -2}r(U_{k},U_{k}) + \sum_{j}\frac{r(U_{j},U_{j})}{n - 2} - \frac{s(m - 1)}{(n-1)(n-2)},
\end{equation}
where $m = n/2$.\\
 \cref{Eq27}, \cref{Eq210} and \cref{Eq34} gives, 
\begin{equation}\label{Eq35}
\begin{aligned}
\sum_{j}\frac{r(U_{j},U_{j})}{n - 2} =& \frac{m}{n -2}r(U_{k},U_{k})+ \frac{s(m -1)}{(n-1)(n-2)}+ g(T^{f}_{U_{k}},T^{f}_{U_{k}})\\& + (m -2)T^{f}_{U_{k}}U_{k}(f)+ |\mathcal{H}(\nabla f)|^{2} + \Delta f|_{\mathcal{H}}- \delta T^{f}(U_{k},U_{k})\\& - g(A^{f}U_{k},A^{f}U_{k}) + m(Hessf(U_{k},U_{k}) + U_{k}(f)^{2}- |\nabla f|^{2}).
\end{aligned}
\end{equation}
Similarly, \cref{Eq21}, \cref{Eq28} and \cref{Eq212} imply
\allowdisplaybreaks{
\begin{equation}\label{Eq36}
\begin{aligned}
\sum_{i}\frac{r(X_{i},X_{i})}{n - 2} =& \frac{m}{n -2}r(X_{l},X_{l})+ \frac{s(m -1)}{(n-1)(n-2)}+ g(T^{f}X_{l},T^{f}X_{l}) -g(A^{f}_{X_{l}},A^{f}_{X_{l}})\\& - |\mathcal{H}(\nabla f)|^{2} - \Delta f|_{\mathcal{H}}+ m(Hessf(X_{l},X_{l}) + X_{l}(f)^{2})\\&-(m-1) |\nabla f|^{2} - g(\nabla_{X_{l}}N^{f},X_{l}) - 2X_{l}(f)g(N^{f},X_{l})  + N^{f}(f)  + \Delta f.
\end{aligned}
\end{equation}}
Adding \cref{Eq35} and \cref{Eq36} gives, 
\allowdisplaybreaks{
   \begin{equation}\label{Eq37}
       \begin{aligned}
           \frac{s}{(n-1)(n-2)}= &\frac{mr(U_{k},U_{k})}{n-2} + \frac{mr(X_{l},X_{l})}{n-2} + g(T^{f}X_{l},T^{f}X_{l})+ g(T^{f}_{U_{k}},T^{f}_{U_{k}})\\&+(m-2)T^{f}_{U_{k}}U_{k}(f) -g(A^{f}_{X_{l}},A^{f}_{X_{l}}) - g(A^{f}U_{k},A^{f}U_{k})+m(Hessf(X_{l},X_{l})\\& + X_{l}(f)^{2}) + m(Hessf(U_{k},U_{k}) + U_{k}(f)^{2}) - (n-1)|\nabla f|^{2}\\& - \delta T^{f}(U_{k},U_{k}) - g(\nabla_{X_{l}}N^{f},X_{l}) - 2X_{l}(f)g(N^{f},X_{l}) + N^{f}(f) + \Delta f.
       \end{aligned}
   \end{equation}}
   Taking summation over $\{U_{k}\}$'s and $\{X_{l}\}$'s,
      \begin{equation}\label{Eq38}
        \begin{aligned}
            \frac{ns}{2(n-1)}=& 2(|A^{f}|^{2} - |T^{f}|^{2} - \hat{\delta}N^{f}) -(n-4)N^{f}(f) - n\Delta f + \frac{n(n-2)}{2}|\nabla f|^{2}.
        \end{aligned}
    \end{equation}
    Proof of $(2)$ and $(3)$ are similar to $(1)$.
\end{proof}
\noindent Employing the divergence theorem and the weak maximal principle in \cref{P31} culminates into the following result.  
\begin{theorem}\label{T32}
A conformal submersion $\pi_{f}\colon (M,g) \rightarrow (B,g_{B})$ as described in \cref{P31}(2) is rigid if either of the following conditions holds.
\begin{enumerate}
\item $T^{f}$ vanishes and $s \leq \sum_{i} \frac{s_{i}}{1-n_{i}}$ where $s_i$ and $n_i$ are as mentioned in \cref{P31}(2).
\item $N^{f}$ is parallel with respect to $g$, the horizontal distribution is integrable and $s \geq \sum_{i} \frac{s_{i}}{1-n_{i}}$.
\end{enumerate} 
Additionally, if $(M,g)$ is locally conformally flat, then $(M,g)$ is scalar flat and locally a Riemannian product.
\end{theorem}
We further observe the following scalar curvature properties for the total space of a conformal submersion given in \cref{T32}. 
\begin{proposition}\label{P33}
Let $\pi_{f}\colon (M,g) \rightarrow (B,g_{B})$ be a conformal submersion such that $(M,g)$ is as defined in \cref{T32}. Moreover, if the scalar curvature of each component of $(M,g)$ have the same sign, then:
\begin{enumerate}
    \item The scalar curvature of $(M,g)$ is non-negative if $T^{f}$ vanishes.
    \item The scalar curvature of $(M,g)$ is non-positive if  $N^{f}$ is parallel with respect to $g$ and the horizontal distribution is integrable.
    \item $(M,g)$ is scalar flat, and $\pi_{f}$ is rigid if the horizontal distribution is integrable and $T^{f}$ vanishes.
\end{enumerate}
\end{proposition}
\begin{proof}
The following is a proof for locally conformally flat total space. The general case follows analogously.\\
    $(1)$ If $T^{f}$ vanishes,
       \begin{equation}\label{Eq39}
        \begin{aligned}
            \frac{ns}{2(n-1)}=& 2|A^{f}|^{2}  - n\Delta f + \frac{n(n-2)}{2}|\nabla f|^{2}.
        \end{aligned}
   \end{equation}
    The scalar curvature of $(M,g)$ is non-negative by the divergence theorem.\\
    $(2)$ If $N^{f}$ is parallel and the horizontal distribution is integrable,
       \begin{equation}\label{Eq310}
        \begin{aligned}
            \frac{ns}{2(n-1)}=& -2|T^{f}|^{2}  -\frac{2(n-4)}{n}|N^{f}|^{2} - n\Delta f + \frac{n(n-2)}{2}|\nabla f|^{2}.
        \end{aligned}
    \end{equation}
    Multiply \cref{Eq310} by $e^{-\frac{n-2}{2}f}$,
       \begin{equation}\label{Eq311}
        \begin{aligned}
            \frac{ns}{2(n-1)}e^{-\frac{n-2}{2}f}=&  - 2e^{-\frac{n-2}{2}f}|T^{f}|^{2}  -\frac{2(n-4)}{n}e^{-\frac{n-2}{2}f}|N^{f}|^{2} - ndiv(e^{-\frac{n-2}{2}f} \nabla f).
        \end{aligned}
    \end{equation}
    The scalar curvature of $(M,g)$ is non-positive by the divergence theorem.\\
    $(3)$ If the horizontal distribution is integrable and $T^{f} = 0$,
     \begin{equation}\label{Eq312}
        \begin{aligned}
            \frac{ns}{2(n-1)}e^{-\frac{n-2}{2}f}=&  - ndiv(e^{-\frac{n-2}{2}f} \nabla f).
        \end{aligned}
    \end{equation}
    $(M,g)$ is scalar flat by the divergence theorem, and $\pi_{f}$ is rigid by the weak maximal principle.
\end{proof}
More generally, if the total space of a conformal submersion is as in \cref{P31}(3), we obtain the following rigidity result similar to \cref{T32}.
\begin{theorem}\label{T34}
A conformal submersion $\pi_{f}\colon (M,g) \rightarrow (B,g_{B})$ as described in  \cref{P31}(3) is
rigid if either of the following conditions holds.
\begin{enumerate}
\item $T^{f}$ vanishes and $s_{M} \leq \sum_{i} \frac{s_{i}}{1-n_{i}}$.
\item $N^{f}$ is parallel with respect to $g$, the horizontal distribution is integrable and $s_{M} \geq \sum_{i} \frac{s_{i}}{1-n_{i}}$.
\end{enumerate} 
where $s_{i}$ is the scalar curvature of the $i^{th}$ component of $(M,g_{M})$ and $s_{M}$ is the scalar curvature of $(M,g_{M})$.
\end{theorem}
In the next two subsections, we prove more rigidity results by putting further conditions on the total space or the base of a conformal submersion as described in this part.  
\subsection{When the total space is quasi-Einstein}\label{s32}
\begin{theorem}\label{T35}
Let $\pi_{f}: (M,g) \rightarrow (B,g_{B})$ be a conformal submersion where $(M,g)$ is as described in \cref{T32} with the scalar curvatures of each component of $(M,g)$ having the same sign. If $(M,g,h,m,\lambda)$ is also quasi-Einstein, then it is Ricci-flat and $\pi_{f}$ is rigid if one of the following conditions holds.
\begin{enumerate}
\item $\lambda$ is non-negative, $N^{f}$ is parallel with respect to $g$, and the horizontal distribution is integrable.
\item $\lambda$ is non-positive, and $T^{f} = 0$.
\item The horizontal distribution is integrable, and $T^{f} = 0$.
\end{enumerate}
Furthermore, if $(M,g)$ is locally conformally flat, then $(M,g)$ is locally a Riemannian product.    
\end{theorem}
\begin{proof}
The following is proof for locally conformally flat total space. The general case follows analogously.\\
We have,
\begin{equation}\label{Eq313}
Ric + Hessh - \frac{1}{m}dh\otimes dh = \lambda g.
\end{equation}
By taking trace, 
\begin{equation}\label{Eq314}
\lambda n-s = \Delta h - \frac{1}{m}|\nabla h|^{2}
\end{equation}
\begin{equation}\label{Eq315}
e^{-h/m}(\lambda n - s) = div(e^{-h/m}\nabla h).
\end{equation}
Using the divergence theorem and \cref{P33}, $\lambda = s = 0$. Applying the weak maximal principle in \cref{Eq314}, $h$ turns out to be constant, and hence $(M,g)$ is Ricci-flat. By \cref{T32}, we can 
conclude that $(M,g)$ is locally a Riemannian product, and $\pi_{f}$ is rigid.
\end{proof}
\subsection{When either the base or the fiber is quasi-Einstein.}\label{s33}
\hfill \break
The following lemma can be derived by using the same method employed to prove \cref{P31} and hence the proof is skipped..
\begin{lemma}\label[lemma]{L36}
Let $\pi_{f}\colon (M,g) \rightarrow (B,g_{B})$ be a conformal submersion such that $(M,g)$ is locally conformally flat and  $dim(B) = dim(F)$ where $F_{p}$ denotes the fiber at $p \in B$. Then,
\begin{equation}\label{Eq316}
\begin{aligned}
e^{-2f}\hat{s} \circ \pi_{f} + \check{s} =& \frac{5n -4}{n}|A^{f}|^{2}- \frac{3n - 4}{n}(|T^{f}|^{2}+ \check{\delta}N^{f}) - 
div(N^{f})\\&-\frac{n^{2}- 12n + 16}{2n}N^{f}(f) - (n-2)div(\mathcal{V}(\nabla f))\\&+\frac{(n-2)(3n-4)}{4}|\mathcal{V}(\nabla 
f)|^{2}. 
\end{aligned}
\end{equation}
\end{lemma}
We further observe some properties of the scalar curvatures of the base and fibers of a conformal submersion as given in \cref{L36}.
\begin{proposition}\label{P37}
Let $\pi_{f}\colon (M,g) \rightarrow (B,g_{B})$ be a conformal submersion as in \cref{L36}. 
If the scalar curvatures of the fibers (respectively, the base) have a particular sign and
\begin{enumerate}
\item $T^{f} = 0$, and the scalar curvature of the base (respectively, the fibers) is non-positive, then the scalar curvatures of the fibers (respectively, the base) are non-negative. 
\item The horizontal distribution is integrable, $T^{f}$ is parallel with respect to $g$, and the scalar curvature of the base (respectively, the fibers) is non-negative, then the scalar curvatures of the fibers(respectively, the base) are non-positive.
\end{enumerate}
\end{proposition}
\begin{proof}
$(1)$. By substituting $T^{f} = 0$ in \cref{Eq316},
\begin{equation}\label{Eq317}
\begin{aligned}
e^{-2f}\hat{s} \circ \pi_{f} + \check{s} =& \frac{5n -4}{n}|A^{f}|^{2} - (n-2)div(\mathcal{V}(\nabla f)) +\frac{(n-2)(3n-4)}{4}|\mathcal{V}(\nabla f)|^{2}. 
\end{aligned}
\end{equation}
Applying the divergence theorem in \cref{Eq317} gives the result.\\
$(2)$. By substituting $A^{f} = 0$ and $\nabla T^{f} = 0$ ($\nabla T^{f} = 0$ implies $|T^{f}|^{2} = N^{f}(f)$) and multiplying $e^{-\frac{3n - 4}{4}f}$ in \cref{Eq316},
\begin{equation}\label{Eq318}
\begin{aligned}
e^{-\frac{3n - 4}{4}f}\hat{s} \circ \pi_{f} + e^{-\frac{3n - 4}{4}f}\check{s} =& - \frac{(n-2)(n-4)}{2n}|T^{f}|^{2}e^{-\frac{3n - 4}{4}f}  - (n-2)(div( e^{-\frac{3n - 4}{4}f}\mathcal{V}(\nabla f)). 
\end{aligned}
\end{equation}
The result follows by applying the divergence theorem in \cref{Eq318}.
\end{proof}
\noindent Employing the divergence theorem and the weak maximal principle in \cref{L36}, \cref{P31}, and then using \cref{P37}, we conclude the following result.
\begin{theorem}\label{T38}
Let $\pi_{f}\colon (M,g) \rightarrow (B,g_{B})$ be a conformal submersion as in \cref{L36}. If the scalar curvatures of the fibers (respectively, the base) have a particular sign and $(F,g_{F},h,m,\lambda)$ (respectively, $(B,g_{B},h,m,\lambda)$) is quasi-Einstein, then $(F,g_{F},h)$ (respectively, $(B,g_{B},h)$) is Ricci-flat, $(B,g_{B})$ (respectively, $(F,g_{F})$) is scalar flat, and also $f$ is constant along fibers whenever either of the following conditions holds.
\begin{enumerate}
\item $T^{f} = 0$, the scalar curvature of the base (respectively, the fibers) is non-positive, and $\lambda \leq 0$.
\item The horizontal distribution is integrable, $T^{f}$ is parallel with respect to  $g$, the scalar curvature of the base (respectively, the fibers) is non-negative, and $\lambda \geq 0$.
\end{enumerate}
In both cases, the horizontal distribution is integrable and $T^{f}$ vanishes. Additionally, if the scalar curvature of $(M,g)$ has a particular sign, then $\pi_{f}$ is rigid.
\end{theorem}
\begin{remark}\label{R39}
Assuming the scalar curvature of the total space has a particular sign is not a strong condition since any closed Riemannian manifold admits a metric in its conformal class that has constant scalar curvature due to the Yamabe problem (see \cite{RSYP}).
\end{remark}
\noindent The following example shows that $dim(F) = dim(B)$ is necessary in \cref{T38}.
\begin{example}\label[example]{E310}
Let $M = S^{1} \times S^{3}$ with metric $g = g_{S^{1}} + e^{2f}g_{S^{3}}$ where $f$ is a non-constant function on $S^{1}$. Consider the map $\pi_{f} \colon (M,g) \rightarrow (S^{3},g_{3})$ defined as $\pi_{f}(x,y) = y$. One can verify by direct computation that $(M,g)$ is locally conformally flat, and all the conditions of \cref{T38} are satisfied by $\pi_{f}$ (note that here, the scalar curvature of $(M,g)$ is constant only up to conformality) except the dimension condition, but $\pi_{f}$ is not rigid.   
\end{example}
\section{Conformal submersions with Einstein total space}\label{s4}
In this section, we analyse the rigidity of conformal submersions where either the total space or its universal cover is the Riemannian product of closed-orientable Einstein manifolds. In both cases, the Ricci tensor turns out to be parallel. This section is divided into three subsections. The first subsection discusses the rigidity when the horizontal distribution is integrable and $N^{f}$ is parallel which is a more general assumption than that of $T^{f}$ being parallel (cf. Lemma $4.6$ of \cite{JML}). In the second subsection, we discuss the rigidity when $N^{f}$ vanishes (a weaker assumption than $T^{f} = 0$, particularly when the dimension of fibers is at least $2$) and the horizontal distribution is not necessarily integrable. 
The last subsection deals with the rigidity when fibers are totally geodesic.
Throughout this section, we will have the following two assumptions (even when it is not explicitly mentioned): (a). the total space or its universal cover is the Riemannian product of closed Einstein manifolds and (b). whenever fibers are assumed to be Einstein, all the fibers share the Einstein constant. 
For the sake of brevity, proofs of all the results in this section are only showcased when the total space is Einstein. However, the same idea also works when either the total space or its universal cover is the Riemannian product of closed-orientable Einstein manifolds. 
\subsection{When the horizontal distribution is integrable and $N^{f}$ is parallel.}\label{s41}
\hfill \break
We now establish \cref{T13}(1), which addresses the rigidity of conformal submersions where the fibers have constant scalar curvature. 
%\begin{theorem}\label{T41}
%Let $\pi_{f} \colon (M,g) \rightarrow (B,g_{B})$ be a conformal submersion with integrable horizontal distribution, and $N^{f}$ is parallel with respect to $g$ and $dim(B) \geq 2$. Then $\pi_{f}$ is rigid if the fibers have constant scalar curvature.
%\end{theorem}
\begin{proof}[Proof of \cref{T13}(1)]
Assume without loss of generality that $(M,g)$ is Einstein with Einstein constant $\lambda$, and $s_{F}$ is the scalar curvature of the fibers. First, we show that $\lambda dim(F) = s_{F}$. The following equation can be derived by using \cref{Eq210}.
\begin{equation}\label{Eq41}
\begin{aligned}
\lambda \, dim(F)  =& s_{F} - |N^{f}|^{2} + (2dim(F)-dim(B))N^{f}(f)\\& + dim(F)(dim(B) -dim(F))|\mathcal{H}(\nabla f)|^{2} - dim(F)(\Delta f)|_{\mathcal{H}} \\& + dim(B)(dim(F) - 1)|\mathcal{V}(\nabla f)|^{2} - dim(B)(\Delta f)|_{\mathcal{V}} 
\end{aligned}
\end{equation}     
Consequently, we observe that $\lambda dim(F) \leq s_{F}$ as $Hess(f)$ is non-negative definite and $\nabla 
f$ vanishes at the point where $f$ attains its global minimum. The equation below can then be derived from \cref{Eq41} by using the fact that $N^{f}$ is parallel. 
\begin{equation}\label{Eq42}
\begin{aligned}
\lambda \, dim(F) =& s_{F}  - dim(F)div(\mathcal{H}(\nabla f))- dim(B)div(\mathcal{V}(\nabla f))\\&   + dim(B)(dim(B) - 1)|\mathcal{V}(\nabla f)|^{2}
\end{aligned}
\end{equation}
Further, using the divergence theorem in \cref{Eq42}, it can be seen that $\lambda dim(F) \geq s_{F}$. Thus, $\lambda dim(F)= s_{F}$. Therefore, $\mathcal{V}(\nabla f)$ vanishes by employing the divergence theorem in \cref{Eq42}. By substituting $\mathcal{V}(\nabla f) = 0$ in \cref{Eq42}, and applying the weak maximal principle, the function $f$ turns out to be constant. This completes the proof.
\end{proof}
\begin{remark}\label{R42}
In particular, \cref{T13}(1) implies that in a conformal submersion with integrable horizontal distribution and parallel $N^{f}$, if the total space and fibers are both Einstein, then their Einstein constants must coincide.
\end{remark}
\subsection{When $N^{f}$ vanishes identically.}\label{s42}
\hfill \break
We now proceed to the proof of \cref{T13}(2a), which concerns the rigidity of conformal submersions with fibers having constant scalar curvature. 
\begin{proof}[Proof of \cref{T13}(2a)]
By using \cref{Eq210},
\begin{equation}\label{Eq43}
\begin{aligned}
\lambda dim(F) - s_{F}=& |A^{f}|^{2} - dim(F)^{2}|\mathcal{H}\nabla f|^{2}- dim(F)\Delta f|_{\mathcal{H}} - dim(B)\Delta f|_{\mathcal{V}}\\&+ dim(F)dim(B)|\nabla f|^{2}  - dim(B)|\mathcal{V}(\nabla f)|^{2}
\end{aligned}
\end{equation}
By using the compactness of $M$, We obtain $|A^{f}|^{2} = \lambda dim(F) - s_{F}$ if $|A^{f}|^{2}$ is constant. By substituting this in \cref{Eq43}, we get
\begin{equation}\label{Eq44}
\begin{aligned}
 dim(F)div(\mathcal{H}(\nabla f)) + dim(B)div(\mathcal{V}(\nabla f))  - dim(B)(dim(B)- 1)|\mathcal{V}(\nabla f)|^{2} =& 0
\end{aligned}
\end{equation}
We conclude the result by applying the divergence theorem and the weak maximal principle in \cref{Eq43} and \cref{Eq44}. 
\end{proof}
The following is the proof of \cref{T13}(2b) which addresses the rigidity conditions when the base is scalar-flat.
\begin{proof}[Proof of \cref{T13}(2b)]
Using \cref{Eq212}, we have
\begin{equation}\label{Eq45}
\begin{aligned}   
\lambda \, dim(B) =& -2|A^{f}|^{2}-|T^{f}|^{2} - (dim(M) - 2)(\Delta f|_{\mathcal{H}} + |\mathcal{H}(\nabla f)|^{2})\\&+ (dim(M) -2)dim(B)|\nabla f)|^{2} -dim(B)\Delta f
\end{aligned}
\end{equation}
Using compactness of $(M,g)$, we can conclude that $\lambda \, dim(B) + 2|A^{f}|^{2} + |T^{f}|^{2} = 0$ if $2|A^{f}|^{2} + |T^{f}|^{2}$ is constant. Substituiting this in \cref{Eq45}, we obtain
\begin{equation}\label{Eq46}
\begin{aligned}   
 (dim(M) - 2)div(\mathcal{H}(\nabla f))- (dim(M) -2)(dim(M) - 1)|\mathcal{H}(\nabla f)|^{2} + dim(B)\Delta f =&0
\end{aligned}
\end{equation}
Consequently, $\pi_{f}$ is rigid by the divergence theorem and the weak maximal principle.
\end{proof}
Using \cref{T13}(1) and \cref{T13}(2b), we conclude the following.
\begin{theorem}\label{T45}
If $(M,g)$ is an Einstein manifold of dimension at least $4$, then up to homothety, $g$ is the unique Einstein metric in its conformal class if $(M,g)$ is the total space of a conformal submersion with integrable horizontal distribution and vanishing $T^{f}$, together with either one-dimensional fibers or a scalar-flat base.  
\end{theorem}
\begin{remark}\label{R46}
In particular, \cref{T45} implies that if $(M,g)$ is Ricci-flat, which is also a Riemannian product, then up to homothety, $g$ is the only Einstein metric in its conformal class.
\end{remark}
\subsection{When the fibers are totally geodesic.}\label{s43}
 \hfill \break
 We now prove \cref{T13}(3), which discusses the rigidity conditions of a conformal submersion with a scalar-flat base.
\begin{proof}[Proof of \cref{T13}(3)]
We use \cref{Eq212} to see that 
\begin{equation}\label{Eq47}
\begin{aligned}
\lambda \, dim(B) = & -2|A^{f}|^{2} - (dim(B)-2)\Delta f|_{\mathcal{H}} -dim(B) \Delta f \\& - (dim(B)-2)|\mathcal{H}\nabla f|^{2}+ dim(B)(dim(B)-2)|\nabla f|^{2} 
\end{aligned}
\end{equation}
Using the compactness of $M$, we obtain $\lambda dim(B) + 2|A^{f}|^{2} = 0$ if $|A^{f}|^{2}$ is constant. By substituting this in \cref{Eq47}, we get
\begin{equation}\label{Eq48}
\begin{aligned}
(dim(B)-2)div(\mathcal{H}(\nabla f)) +dim(B) \Delta f  - (dim(B)-2)(dim(B)-1)|\mathcal{H}(\nabla f)|^{2}  = & 0
\end{aligned}
\end{equation}
If $dim(B) = 2$, then by the weak maximal principle, $f$ is constant. \\
When $dim(B) > 2$, using the divergence theorem in \cref{Eq48}, it follows that $\mathcal{H}(\nabla f) = 0$. Combining this with \cref{Eq48} and applying the weak maximal principle, $f$ becomes constant.
\end{proof}
\section{Rigidity of quasi-Einstein manifolds}\label{S5}
In this section, we conclude rigidity results for certain quasi-Einstein manifolds that can be realised as the total space, the fiber or the base of suitable conformal submersions. All the quasi-Einstein manifolds mentioned here are closed and orientable. The first result is about a quasi-Einstein manifold realised as the fiber or the base of a conformal submersion.
\begin{theorem}\label{T51}
Let $(M,g,h,m,\lambda)$ be a quasi-Einstein manifold with $\lambda \geq 0$ and $m \geq 1$. Then $(M,g,h)$ is Ricci flat if $(M,g)$ is the fiber (respectively, the base) of a conformal submersion with an integrable horizontal distribution and parallel $T^{f}$ satisfying the following conditions. 
\begin{enumerate}
\item The total space of the conformal submersion is locally conformally flat, and its dimension is twice the dimension of $M$.
\item The base (respectively, the fiber) has non-negative scalar curvature.
\end{enumerate}
\end{theorem}
\begin{proof}
Proposition $3.6$ of \cite{RQE}  says that`\textit{If $(M,g,h,m,\lambda)$ is quasi-Einstein with $m \geq 1$, $\lambda > 0$ and $M$ is compact, then the scalar curvature of $(M,g)$ is bounded below by} $\frac{n(n-1)\lambda}{m + n -1}$'. It readily follows from the above result that the scalar curvature of $(M,g)$ (as in the hypothesis) is non-negative if $\lambda \geq 0$. Using \cref{T38}, we conclude the result.
\end{proof}
The next result is about a locally conformally flat quasi-Einstein manifold realised as the total space of a conformal submersion.
\begin{theorem}\label{T52}
If $(M,g,h,m,\lambda)$ is a locally conformally flat quasi-Einstein manifold of dimension at least 4 with $m \geq 1$ and $\lambda \geq 0$, then $(M,g,h)$ is a flat space form if it is the total space of a conformal submersion with an integrable horizontal distribution and parallel $N^{f}$ whose base and fiber have the same dimension.
\end{theorem}
\begin{proof}
Again, by using Proposition $3.6$ of \cite{RQE}, the scalar curvature of $(M,g)$ is non-negative. Finally, \cref{T35} concludes the result.
\end{proof}
\begin{example}\label[example]{E53}
Let $(M_{1} \times M_{2},g)$ be a locally conformally flat manifold with $dim(M_{1}) = dim(M_{2}) \geq 2$, and $g$ be the product metric. Then $(M_{1} \times M_{2},g)$ is a flat space form if there exist functions $f,h \in C^{\infty}(M_{1} \times M_{2})$ such that $(M_{1} \times M_{2},e^{2f}g,h,m,\lambda)$ is quasi-Einstein with $m \geq 1$ and $\lambda \geq 0$. 
\end{example}
\noindent \cref{T14}(2) follows from \cref{T13}(1) together with \cref{Eq212}. Likewise, \cref{T14}(3) follows from \cref{T13}(2a) and \cref{Eq212}.\\

Using \cref{T13}(2b) and \cref{Eq210}, we get
\begin{theorem}\label{T56}
    A quasi-Einstein manifold is rigid if it is the fiber of a conformal submersion with vanishing $N^{f}$ and constant $|A^{f}|^{2}$ and $|T^{f}|^{2}$, where
    \begin{enumerate}
        \item The total space is a product of Einstein manifolds
        \item The base is scalar flat.
    \end{enumerate}
\end{theorem}
Finally, we obtain the following rigidity result using \cref{T13}(3) and \cref{Eq210}.
\begin{theorem}\label{T57}
    A quasi-Einstein manifold is rigid if it is the fiber of a conformal submersion with totally geodesic fibers and constant $|A^{f}|^{2}$, where
    \begin{enumerate}
        \item The total space is a product of Einstein manifolds
        \item The base is scalar flat.
    \end{enumerate}
\end{theorem}
\bibliographystyle{amsplain}
\bibliography{references}

\end{document}